\documentclass[11pt,reqno]{amsart}
\usepackage{amsfonts, amsbsy, amsmath, amssymb, mathtools, amsthm,ragged2e}
\usepackage{amscd}
\usepackage{color,enumerate}
\usepackage{breqn}
\usepackage{bm}
\usepackage{pgfplots}
\pgfplotsset{compat=1.15}
\usepackage{tikz}
\usepackage{tikz-cd}
\usepackage{subcaption}
\usetikzlibrary{arrows}
\usepackage{verbatim}
\usepackage{graphicx}
\usetikzlibrary{matrix,calc}
\usepackage{caption}
\usepackage{hyperref}
\usepackage{float}
\usepackage[mathscr]{euscript}
\usepackage{mathrsfs}
\usepackage[margin=2.5cm]{geometry}
\usepackage{epsfig}

\newtheorem{thm}{Theorem}[section]
\newtheorem{lemma}[thm]{Lemma}

\newtheorem{prop}[thm]{Proposition}
\newtheorem{conj}[thm]{Conjecture}
\newtheorem{thm-con}[thm]{Theorem-Conjecture}
\numberwithin{equation}{section}

\theoremstyle{definition}
\newtheorem{defn}[thm]{Definition}
\newtheorem{rmk}[thm]{Remark}

\newtheorem{exmp}[thm]{Example}

\DeclareMathOperator{\pol}{pol}

\DeclareMathOperator{\Ass}{Ass}
\DeclareMathOperator{\MinAss}{MinAss}

\newcommand{\G}{\mathcal{G}}

\newcommand{\m}{\mathfrak{m}}

\def\K{\mathbb{K}}

\begin{document}
\title[Support-2 Monomial Ideals that are Simis]{Support-2 Monomial Ideals that are Simis}
\author[Paromita Bordoloi]{Paromita Bordoloi}
\email{2022rma0026@iitjammu.ac.in}
\address{Department of Mathematics, Indian Institute of Technology Jammu, J\&K, India - 181221.}

\author[Kanoy Kumar Das]{Kanoy Kumar Das}
\email{kanoydas0296@gmail.com, kanoydas@cmi.ac.in}
\address{Chennai Mathematical Institute, H1, SIPCOT IT Park, Siruseri, Kelambakkam 603103, India.}

\author[Rajiv Kumar]{Rajiv Kumar}
\email{rajiv.kumar@iitjammu.ac.in}
\address{Department of Mathematics, Indian Institute of Technology Jammu, J\&K, India - 181221.}

\date{\today}

\subjclass[2020]{Primary 13C70, 05E40; Secondary 13A70, 05C22, 05C25}

\keywords{support-$2$ monomial ideal, symbolic power, Simis ideal, primary decomposition}

\begin{abstract}
    A monomial ideal $I\subseteq \K[x_1,\ldots , x_n]$ is called a Simis ideal if $I^{(s)}=I^s$ for all $s\geq 1$, where $I^{(s)}$ denotes the $s$-th symbolic power of $I$. Let $I$ be a support-$2$ monomial ideal such that its irreducible primary decomposition is minimal.  We prove that $I$ is a Simis ideal if and only if $\sqrt{I}$ is Simis and $I$ has a standard linear weighting. This result thereby proves a recent conjecture for the class of support-$2$ monomial ideals proposed by Mendez, Pinto, and Villarreal. 
    Furthermore, we give a complete characterization of the Cohen-Macaulay property for support-$2$ monomial ideals whose radical is the edge ideal of a whiskered graph. Finally, we classify when these ideals are Simis in degree 2.
\end{abstract}

\maketitle
\section{Introduction and Preliminaries}\label{Intro_Pre}
Let $S=\K[x_1,\ldots ,x_n]$ be a polynomial ring over a field $\K$ and $I\subseteq S$ be an ideal. The $s$-th \emph{symbolic power} of $I$, denoted by $I^{(s)}$, is defined by 
\[
I^{(s)}= \bigcap_{P\in \MinAss(I)} (I^sS_P\cap S),
\]
where $\MinAss(I)$ denotes the set of all minimal primes of $S/I$. 
One has the following simpler expression \cite[Lemma 2]{gimenez2018symbolic} for the $s$-th symbolic power of a monomial ideal $I$:
\[
I^{(s)} = \bigcap_{P \in \MinAss(I)}Q(P)^s,
\]
where $Q(P)$ is the primary component of $I$ corresponding to $P$.
Notably, there is an alternative notion of symbolic powers, defined using the entire set of associated primes $\Ass(I)$ of the ideal $I$ (see \cite{CooperEtAl2017}, or the survey \cite{SymbolicSurvey2018}). These two notions of symbolic powers coincide if $\Ass(I)=\MinAss(I)$, that is, when $I$ does not have any embedded associated prime. 

\par 
A general theme of study in the context of symbolic powers is the comparison of symbolic powers with the ordinary powers of an ideal. The so-called \textit{containment problem} asks one to determine positive integers $s$ and $r$ for which the containment $I^{(s)}\subseteq I^r$ holds. Another quite natural question that one can ask is when does the $s$-th symbolic power coincide with the $s$-th ordinary power? These two problems have been extensively investigated for several important classes of ideals from various aspects of mathematics. In this work, we address these two questions for the class of monomial ideals in a polynomial ring, using commutative algebra and combinatorics.

\par
The term `Simis ideal' is introduced to recognize the pioneering work of Aron Simis on symbolic
powers of monomial ideals.  A monomial ideal $I$ is called a \textit{Simis ideal} if $I^{(s)}=I^s$ for all $s\geq 1$, and is called \textit{Simis in degree $s$} if $I^{(s)}=I^s$. Finding a complete combinatorial characterization of Simis ideals is an extremely difficult task. For square-free monomial ideals, this problem is equivalent (see \cite{GRV2009}) to a famous conjecture about the packing problem due to Conforti and Cornuejols \cite{Packing}. A classical result due to Simis, Vasconcelos, and Villarreal \cite{SimisVasconcelosVillarreal1994} characterizes edge ideals of simple graphs that are Simis ideals. There are other classes of monomial ideals for which Simis ideals have been identified: the edge ideals of weighted oriented graphs \cite{GMV2024, MandalPradhan2021}, a certain class of generalized edge ideals \cite{das2024equality}, cover ideals of graphs \cite{GRV2005, GRV2009}, matroidal ideals \cite{ficarra2025symbolic}, to name a few.

\par
We now set up proper notations and conventions that we want to follow in the rest of the article. The objects we are going to study are the monomial ideals whose minimal generating set consists of monomials of support $2$. To be more precise, let $S=\K[x_1,\ldots , x_n]$ be a polynomial ring, where $\K$ is a field. Recall that monomial ideals are the ideals in $S$ which are generated by monomials. We say that an ideal $I\subseteq S$ is a \textit{support-$2$ monomial ideal} if $\mathcal{G}(I)\subseteq \{x_i^ax_j^b\mid 1\leq i<j\leq n, a,b\geq 1\}$, where $\mathcal{G}(I)$ denotes the unique set of minimal monomial generators of the ideal $I$. A classical example of support-$2$ monomial ideals is the class of edge ideals of graphs, which have been studied extensively in the literature. Let $G$ be a simple graph with the vertex set $V(G)=\{x_1,\ldots , x_n\}$, and the edge set $E(G)\subseteq 2^{V(G)}$. The \textit{edge ideal} of $G$, denoted by $I(G)$, is defined by $I(G):=\left\langle x_ix_j \mid \{x_i,x_j\}\in E(G)\right\rangle$. By abuse of notation, we sometimes write $\{i,j\}\in E(G)$, instead of $\{x_i,x_j\}\in E(G)$. When we say $\{i,j\}\in E(G)$, or $\{x_i,x_j\}\in E(G)$, or $x_ix_j\in I(G)$, or $x_i^ax_j^b\in \G(I)$ for some support-$2$ monomial ideal $I$, we always assume that $i<j$, and we follow this convention throughout this article. Note that, if $I$ is a support-$2$ monomial ideal, then its radical $\sqrt{I}$ is an edge ideal of some graph $G$, which we will denote by $G(I)$, and call this the \textit{underlying simple graph} of the ideal $I$.

Let $I\subseteq S$ be a support-$2$ monomial ideal, and $G(I)$ be the underlying simple graph of $I$. For an edge $\{x_i,x_j\}\in E(G(I))$ with $i<j$, let 
\[
T_{i,j}=\{x_i^{w_{i,j}^1}x_j^{w_{j,i}^1},\ldots , x_i^{w_{i,j}^k}x_j^{w_{j,i}^k}\}\subseteq \G(I)
\] 
be the set of all minimal monomial generators of $I$ supported by the variables $x_i,x_j$, where $k\geq 1$. If $\{x_i,x_j\}\notin E(G(I))$, then we set $T_{i,j}=\emptyset$. Thus, we have a partition of the set of minimal monomial generators of $I$, given by \[\G(I)=\bigcup_{\substack{\{x_i,x_j\}\in E(G(I)),\\ i<j}}T_{i,j}.\] Moreover, for our purpose, we assume that
\[
w_{i,j}^1>w_{i,j}^2>\cdots > w_{i,j}^k,
\]
and this will imply that 
\[
w_{j,i}^1<w_{j,i}^2<\cdots < w_{j,i}^k.
\]
With the aid of the above conventions, for a given support-$2$ monomial ideal $I\subseteq S$ and for $1\leq i<j\leq n$ such that $\{x_i,x_j\}\in E(G(I))$, we fix the following notation:
\begin{itemize}
    \item $\alpha_{i,j}$ is the number of minimal monomial generator(s) of $I$ whose support is $\{x_i,x_j\}$. We follow the convention that $\alpha_{i,j}=0$ if $\{x_i,x_j\}\notin E(G(I))$.

    \item $\mu_{i,j}:=w_{i,j}^1$, $\mu_{j,i}:=w_{j,i}^{\alpha_{i,j}}$. In other words, for any $i$, $\mu_{i,j}$ denotes the maximum exponent of the variable $x_i$ in the minimal monomial generators of $I$ whose support is $\{x_i,x_j\}$.

    \item $\nu_{i,j}:=w_{i,j}^{\alpha_{i,j}}$, $\nu_{j,i}:=w_{j,i}^{1}$. In other words, for any $i$, $\nu_{i,j}$ denotes the minimum exponent of the variable $x_i$ in the minimal monomial generators of $I$ whose support is $\{x_i,x_j\}$.
\end{itemize}

We illustrate the above notation with the help of the following example. Consider the ideal, $I=\left\langle x_1x_2^2,x_2^3x_3,x_2^2x_3^2,x_2x_3^3,x_1x_3^2 \right\rangle\subseteq \K[x_1,x_2,x_3]$. Here, $G(I) = C_3$, the cycle graph on $3$ vertices. For the edges $\{x_1,x_2\}$ and $\{x_1,x_3\}$, we see that $\alpha_{1,2}=1=\alpha_{1,3}$. On the other hand, for the edge $\{x_2,x_3\}$, we have $\alpha_{2,3} = 3$. Moreover, $\mu_{2,3}=w^1_{2,3}=3, w^2_{2,3}=2, \nu_{2,3}=w^3_{2,3}=1, \nu_{3,2}=w^1_{3,2}=1, w^2_{3,2}=2, \mu_{3,2}=w^3_{3,2}=3$. If $\alpha_{i,j}=1$ for all $\{x_i,x_j\}\in E(G(I))$, then we say that $I$ is a \textit{generalized edge ideal} (see \cite[Definition 2.1]{das2024equality}). Note that the class of edge ideals of weighted oriented graphs is a subclass of the class of generalized edge ideals.

We now recall the notion of weighted monomial ideals as introduced in \cite{mendez2024symbolic}.  The monomials of $S$ are denoted by $\mathbf{x}^{\mathbf{a}}:=x_1^{a_1}\cdots x_n^{a_n}$, where $\mathbf{a}=(a_1,\ldots ,a_n)$. Let $w:\mathbb{R}^n\rightarrow \mathbb{R}^n$ be a linear function such that $w(\mathbb{N}_+^n)\subseteq \mathbb{N}_+^n$, where $\mathbb{N}_+$ denotes the set of all positive integers. Note that the function $w:\mathbb{R}^n\rightarrow \mathbb{R}^n$ can be expressed as $w=(w_1,\ldots , w_n)$, where $w_i:\mathbb{R}^n\rightarrow \mathbb{R}$ is the $i$-th coordinate function of $w$, that is, $w(\mathbf{a})=(w_1(\mathbf{a}),\ldots , w_n(\mathbf{a}))$ for all $\mathbf{a}\in \mathbb{R}^n$. We call the function $w$ a {\textit{linear weighting}} of $S$. A linear weighting $w$ of $S$ is called \textit{{standard}} if there are positive integers $d_1,\ldots , d_n$ such that for all $\mathbf{a}=(a_1,\ldots , a_n)\in \mathbb{R}^n$, $w(\mathbf{a})=(d_1a_1, \ldots ,d_na_n)$. Note that, in the case of standard linear weighting, the function $w:\mathbb{R}^n\rightarrow \mathbb{R}^n$ is completely determined by $d_1,\ldots , d_n$, as $w(\mathbf{e}_i)=d_i\mathbf{e}_i$ for all $1\leq i\leq n$, where $\mathbf{e}_i, 1\leq i\leq n$ are the unit vectors in $\mathbb{R}^n$. 

Let $I\subseteq S$ be a monomial ideal and let $\mathcal{G}(I)=\{\mathbf{x}^{\mathbf{b}_1}, \ldots , \mathbf{x}^{\mathbf{b}_t}\}$. The weighted monomial ideal of $I$ with respect to a linear weighting $w$, denoted by $I_w$, is defined by 
\[
I_w:=\left\langle\mathbf{x}^{w(\mathbf{a})}\mid \mathbf{x}^{\mathbf{a}}\in I \right\rangle=\left\langle \mathbf{x}^{w(\mathbf{b}_i)}\mid 1\leq i\leq t\right\rangle.
\]

When $w$ is a standard linear weighting, then the weighted monomial ideal $I_w$ is generated minimally by the monomials obtained from $\mathcal{G}(I)$ after replacing each $x_i$ with $x_i^{d_i}$, where $w(\mathbf{e}_i)=d_i\mathbf{e}_i, 1\leq i\leq n$.

\par
A primary decomposition $I=\bigcap_{i=1}^r Q_i$ of an ideal $I$ is said to be \textit{irredundant} if none of the ideals $Q_i$ can be omitted from the intersection. Throughout this article, we assume that every decomposition of an ideal is irredundant. Such a decomposition is called minimal if $\sqrt{Q_i}\neq \sqrt{Q_j}$ for $i\neq j$. An ideal $Q\subseteq S$ is called \textit{irreducible} if it cannot be written as an intersection of two ideals of $S$ that properly contain $Q$. It can be deduced that irreducible monomial ideals are the ideals of the form $\left\langle x_i^{a_{i}}\mid i\in [n], a_{i}\geq 1 \right\rangle$. Moreover, irreducible ideals are primary, and any monomial ideal $I \subseteq S$ has a unique irredundant irreducible decomposition.
Although the irreducible decomposition of a monomial ideal $I$ need not be minimal, there are certain classes of ideals (e.g., square-free monomial ideals, edge ideals of weighted oriented graphs) such that their irreducible decompositions are always ~minimal. 
\par
In order to understand the class of Simis ideals, the authors in \cite{mendez2024symbolic} made the following ~conjecture.

\begin{conj}\cite[Conjecture 5.7]{mendez2024symbolic}\label{conj}
    Let $I$ be a monomial ideal without any embedded associated primes. If the irreducible decomposition of $I$ is minimal, and $I$ is a Simis ideal, then there is a Simis square-free monomial ideal $J$ and a standard linear weighting $w$ such that $I=J_w$.   
\end{conj}

This conjecture is known to hold only for edge ideals of simple graphs \cite{SimisVasconcelosVillarreal1994}, and the edge ideals of weighted oriented graphs \cite[Theorem 3.3]{GMV2024}. One of the main contributions of this article is the verification of this conjecture for the class of support-$2$ monomial ideals. 

\medskip
\noindent
\textbf{Theorem~~\ref{thm: conj main thm}.} 
    Let $I\subseteq S$ be a support-$2$ monomial ideal. Then Conjecture~\ref{conj} is true for $I$.

\medskip
Next, we consider a few specific classes of support-$2$ monomial ideals and try to understand their small symbolic powers. We show that when $G(I)=C_n$, where $n=4$ or $n\geq 6$, then $I$ is Simis if and only if $n$ is even and $I$ has a standard linear weighting (see Theorem~\ref{thm: Simis cycle} and Remark~\ref{remark: 4 cycle}). Surprisingly, this characterization is not true when $n=3,5$. Another outcome of the study of small symbolic powers is the following theorem, where we classify Cohen-Macaulay support-$2$ monomial ideals whose radical is an edge ideal of a whiskered graph. We show the following:

\medskip
\noindent
\textbf{Theorem~\ref{thm: CM whisker}.}
    Let $I\subseteq \K[x_1,\ldots ,x_m,x_{m+1},\ldots , x_{2m}]$ be a support-$2$ monomial ideal such that  $G(I)$ is the whiskered graph $W_H$ on the simple graph $H$, where $V(H)=\{x_1,\ldots , x_m\}$ and, $E(G(I))=E(H)\cup \{\{x_1,x_{m+1}\}, \ldots , \{x_m,x_{2m}\}\}.$ 
    Then, the following statements are equivalent:
    \begin{enumerate}\rm
        \item $S/I$ is Cohen-Macaulay,
        \item $I$ is unmixed,
        \item $I$ does not have any embedded primes,
        \item For all $1\leq i\leq m$, $\alpha_{i,m+i}=1$ and $w^1_{i,m+i} \geq \mu_{i,j}$ for any $1\leq j\leq m$.
    \end{enumerate}

\medskip
Finally, we give a complete combinatorial characterization when the second symbolic power of a special class of generalized edge ideals coincides with the second ordinary power.

\medskip
\noindent
\textbf{Theorem~\ref{thm:whisker second power}}.
    Let $I$ be a support-$2$ monomial ideal such that $G(I)$ is the whiskered graph $W_H$. Let $V(H)=\{x_1,\ldots , x_m\}$ and $E(G(I))=E(H)\cup \{\{x_1,x_{m+1}\}, \ldots , \{x_m,x_{2m}\}\}$. Further, assume that $I$ does not have any embedded prime, $G(I)$ is triangle-free, and $\alpha_{i,j}\leq 1$ for all $1\leq i,j\leq 2m$. Then, the following statements are equivalent:
     \begin{enumerate}\rm
         \item $ I^{(2)} = I^2$,
         \item For each edge $\{x_i,x_j\} \in E(H)$, one of the following conditions is satisfied:
         \begin{itemize}
             \item $w^1_{i,m+i}=w^1_{i,j}$ and $w^1_{j,m+j}=w^1_{j,i}$,
             \item $w^1_{i,m+i}\geq 2w^1_{i,j}$  and $w^1_{j,m+j} \geq 2w^1_{j,i}$.
         \end{itemize}
     \end{enumerate}

\medskip
This paper is organized as follows. In Section~\ref{sec: MPV conj}, we prove the Mendez-Pinto-Villarreal conjecture for support-$2$ monomial ideals. In Section~\ref{sec: small powers}, we investigate the Simis property for some classes of support-$2$ monomial ideals. As a consequence, we characterize the Cohen-Macaulayness of a special class of support-$2$ monomial ideals. Finally, in Section~\ref{sec:discussion}, we present some interesting examples that we have encountered while studying these ideals.

\section{The Mendez-Pinto-Villarreal Conjecture}\label{sec: MPV conj}

The main content of this section is the verification of Conjecture~\ref{conj} for support-$2$ monomial ideals. We first prove a few auxiliary lemmas, which are the crux of the proof of this conjecture. We will continue to use the notations and conventions that we have introduced in Section \ref{Intro_Pre}.

\begin{lemma}\label{lem:conj lem 1}
    Let $I$ be a Simis support-$2$ monomial ideal such that the irreducible decomposition of $I$ is minimal. Assume that $\alpha_{i,j}=1$ for all $\{x_i,x_j\}\in E(G(I))$. Then there is a Simis square-free monomial ideal $J$ and a standard linear weighting $w$ such that $I=J_w$.
\end{lemma}
\begin{proof}
    For simplicity, we write $w_{i,j}$ instead of $w_{i,j}^1$ throughout the proof. By way of contradiction, let us assume that $I\neq J_w$ for any standard linear weighting $w$. Then there exists a vertex $x_i\in V(G(I))$ such that $w_{i,j}\neq w_{i,k}$ for some edges $\{x_i,x_j\}, \{x_i,x_k\}\in E(G(I))$. Without loss of generality, we assume that $w_{i,j}>w_{i,k}$, and set 
    \[
    f=x_i^{\max \{w_{i,j}, 2w_{i,k}\}}x_j^{w_{j,i}} x_k^{w_{k,i}}.
    \]
    We claim that $f\in I^{(2)}$. Indeed, we claim that for every minimal prime $P$ of $I$, $f\in Q(P)^2$, where $Q(P)$ is the primary component of $I$ corresponding to $P$. R2C4: We proceed by considering the following cases. First, assume that $x_{i}^{w_{i,k}}\in Q(P)$. Then it is evident that $f\in Q(P)^2$. Next, if $x_{i}^{w_{i,k}}\notin Q(P)$, since the irreducible decomposition of $I$ is minimal, it follows that $x_{k}^{w_{k,i}}\in Q(P)$. Now, if $x_{i}^{w_{i,j}}\in Q(P)$, then again we conclude that $f\in Q(P)^2$. On the other hand, if $x_{i}^{w_{i,j}}\notin Q(P)$, then by the same reason, we have $x_{j}^{w_{j,i}}\in Q(P)$, and hence, $f\in Q(P)^2$. Therefore, in any case $f\in Q(P)^2$ and this completes the proof of our claim.


    Next, we are going to show that $I^{(2)}\neq I^2$. For this, we consider the following two cases:

    \noindent \textsc{Case I}: Assume that the induced subgraph of $G(I)$ on the vertices $x_i,x_j,x_k$ is not the $3$-cycle $C_3$. Then, $\{x_j,x_k\}\notin E(G(I))$. We claim that $f\notin I^2$. Indeed, if $f\in I^2$, then either one of the following should hold: (A) $(x_i^{w_{i,j}}x_j^{w_{j,i}})^2\mid f$, (B) $(x_i^{w_{i,j}}x_j^{w_{j,i}})\cdot (x_i^{w_{i,k}}x_k^{w_{k,i}}) \mid f$, (C) $(x_i^{w_{i,k}}x_k^{w_{k,i}})^2\mid f$. Note that, (A) is an impossibility, as $x_i^{2w_{i,j}}\nmid f$; (B) is also not possible, since $x_i^{w_{i,j}+w_{i,k}}\nmid f$; and (C) is not true because of $x_k^{2w_{k,i}}\nmid f$.

    \noindent \textsc{Case II}: Assume that  the vertices $x_i,x_j,x_k$ form a $3$-cycle in $G(I)$. Thus, in this case we have $\{x_j,x_k\}\in E(G(I)).$ Now, to see whether $f\in I^2$ or not, we have to consider the following in addition to \textsc{Case I}: (D) $(x_j^{w_{j,k}}x_k^{w_{k,j}})^2\mid f$, (E) $(x_i^{w_{i,j}}x_j^{w_{j,i}})\cdot (x_j^{w_{j,k}}x_k^{w_{k,j}}) \mid f$, (F) $(x_i^{w_{i,k}}x_k^{w_{k,i}})\cdot (x_j^{w_{j,k}}x_k^{w_{k,j}})\mid f$. Note that both (E) and (F) are impossibilities, as $x_j^{w_{j,i}+w_{j,k}}\nmid f$ and $x_k^{w_{k,i}+w_{k,j}}\nmid f$. Therefore, we can conclude that $f\in I^2$ if and only if (D) holds, that is, $(x_j^{w_{j,k}}x_k^{w_{k,j}})^2\mid f$. Now, if $(x_j^{w_{j,k}}x_k^{w_{k,j}})^2\nmid f$ then $f\notin I^2$ and we are through. Otherwise, if $(x_j^{w_{j,k}}x_k^{w_{k,j}})^2\mid f$, then we have $w_{j,i}\geq 2w_{j,k}$ and $w_{k,i}\geq 2w_{k,j}$. In this case, since $w_{k,i}>w_{k,j}$, we consider the monomial 
    \[
    f'=x_k^{\max \{w_{k,i}, 2w_{k,j}\}}x_i^{w_{i,k}} x_j^{w_{j,k}}.
    \]
    Then, by the claim above, it follows that $f'\in I^{(2)}$. On the other hand, by similar arguments as before, $f'\in I^2$ if and only if $(x_i^{w_{i,j}}x_j^{w_{j,i}})^2\mid f'$. But this is an impossibility, as $w_{i,k}< w_{i,j}$. This completes the proof of the fact that $I^{(2)}\neq I^2$, which is a contradiction to the hypothesis that $I$ is~ Simis. 
\end{proof}
    
\begin{lemma}\label{lem:conj lem 2}
    Let $I\subseteq S$ be a support-$2$ monomial ideal such that the irreducible decomposition of $I$ is minimal. If  $\alpha_{i,j}\geq 2$  for some $1\leq i,j\leq n$, then $I^{(2)}\neq I^2$.
\end{lemma}
\begin{proof}
   For simplicity, we set $\alpha_{i,j}=k$ where $\{x_i,x_j\}\in E(G(I))$. Take $$T_{i,j}=\{x_i^{w_{i,j}^1}x_j^{w_{j,i}^1},\ldots , x_i^{w_{i,j}^k}x_j^{w_{j,i}^k}\}\subseteq \G(I)$$ to be the set of all minimal monomial generators whose support is $\{x_i,x_j\}$. We also assume that $w_{i,j}^1>w_{i,j}^2>\cdots > w_{i,j}^k$ and $w_{j,i}^1<w_{j,i}^2<\cdots < w_{j,i}^k$. Now, consider the monomial 
    \[
    f=x_i^{w_{i,j}^1+w_{i,j}^2}x_j^{\max \{2w_{j,i}^1, w_{j,i}^2\}}.
    \]
    We claim that $f\in I^{(2)}\setminus I^2$. To see $f\in I^{(2)}$, consider any minimal prime ideal $P$ of the ideal $I$, and $Q(P)$ be the primary component corresponding to $P$. First, assume that $x_i\notin P$. As $x_i^{w_{i,j}^1}x_j^{w_{j,i}^1}\in \mathcal{G}(I)$, we have $x_j^{w_{j,i}^1}\in Q(P)$. Note that $x_j^{2w_{j,i}^1}\mid f$, and hence $f\in Q(P)^2$. Secondly, if $x_j\notin P$, then since $x_i^{w_{i,j}^2}x_j^{w_{j,i}^2}\in \mathcal{G}(I)$, we have $x_i^{w_{i,j}^2}\in Q(P)$. Now, $w_{i,j}^1+w_{i,j}^2 > 2w_{i,j}^2$, and therefore, $x_i^{2w_{i,j}^2}\mid f$, which in turn, implies that $f\in Q(P)^2$. Finally, assume that both $x_i, x_j \text{ are in } P$. Since the irreducible decomposition of $I$ is minimal, $x_i^{w_{i,j}^\ell}x_j^{w_{j,i}^\ell}\notin \mathcal{G}(Q(P))$ for all $1\leq \ell\leq k$.
    Let $1\leq t\leq k$ be the smallest integer such that $x_i^{w_{i,j}^{t}}\notin Q(P)$. Firstly, if $t=1$, then  $x_i^{w_{i,j}^1}\notin Q(P)$. Since $Q(P)$ is irreducible, $x_j^{w_{j,i}^1}\in Q(P)$. As $x_j^{2w_{j,i}^1}\mid f$, we have $f\in Q(P)^2$. Secondly, if $t=2$, then $x_i^{w_{i,j}^1}\in Q(P)$, but $x_i^{w_{i,j}^2}\notin Q(P)$. Then since $x_i^{w_{i,j}^2}x_j^{w_{j,i}^2}\in \mathcal{G}(I)$ and $Q(P)$ is irreducible, $x_j^{w_{j,i}^2}\in Q(P)$. Note that, $x_i^{w_{i,j}^1}\cdot x_j^{w_{j,i}^2}\mid f$ and therefore, $f\in Q(P)^2$. Finally, if $t\geq 3$, then $x_i^{w_{i,j}^{t}}\notin Q(P)$ and $x_i^{w_{i,j}^{t-1}}\in Q(P)$. Note that for $t\geq 3$, $w_{i,j}^{1}+w_{i,j}^{2}> 2w_{i,j}^{t-1}$. This implies that $x_i^{2w_{i,j}^{t-1}}\mid f$, and hence, $f\in Q(P)^2$. This concludes the proof of our claim.

    It now remains to show that $f\notin I^2$. We have two possible choices of the monomial $f$, namely, (A) $f=x_i^{w_{i,j}^1+w_{i,j}^2}x_j^{2w_{j,i}^1}$, (B) $f=x_i^{w_{i,j}^1+w_{i,j}^2}x_j^{w_{j,i}^2}$. 
    In each of the above cases, it is straightforward to verify that $f\notin I^2$. Indeed, if $f=x_i^{w_{i,j}^1+w_{i,j}^2}x_j^{2w_{j,i}^1}$, then since the exponent of $x_j$ in $f$ is $2w_{j,i}^1$, and $w_{j,i}^1< w_{j,i}^\ell$ for all $2\leq \ell\leq k$, $f\in I^2$ only when $(x_i^{w_{i,j}^1}x_j^{w_{j,i}^1})^2\mid f$. But, $x_i^{2w_{i,j}^1}\nmid f$ as $w_{i,j}^1+w_{i,j}^2< 2w_{i,j}^1$. Secondly, if $f=x_i^{w_{i,j}^1+w_{i,j}^2}x_j^{w_{j,i}^2}$, then observe that $x_i^{w_{i,j}^\ell}x_j^{w_{j,i}^\ell}\mid f$ if and only if $\ell=1,2$. Therefore, if $g_1g_2\mid f$ for some $g_1,g_2\in \mathcal{G}(I)$, then it follows that $g_1,g_2\in \{x_i^{w_{i,j}^1}x_j^{w_{j,i}^1}, x_i^{w_{i,j}^2}x_j^{w_{j,i}^2}\}$. Since the exponent of the variable $x_j$ in $f$ is $w_{j,i}^2$ and $w_{j,i}^1<w_{j,i}^2$, it follows that $g_1=g_2=x_i^{w_{i,j}^1}x_j^{w_{j,i}^1}$. Again, this is a contradiction since $x_i^{2w_{i,j}^1}\nmid f$. 
\end{proof}

It is to be noted that Lemma \ref{lem:conj lem 1} and Lemma \ref{lem:conj lem 2} do not hold if one drops the condition that the irreducible decomposition is minimal (see Example \ref{example:1} and Example \ref{example:5}).

\begin{thm}\label{thm: conj main thm}
     Let $I$ be a support-$2$ monomial ideal such that the irreducible decomposition of $I$ is minimal and $I$ does not have any embedded primes. Then $I$ is a Simis ideal if and only if $G(I)$ is a bipartite graph and $I$ has a standard linear weighting. In particular, if $I\subseteq S$ is a support-$2$ monomial ideal, then Conjecture~\ref{conj} is true for $I$.
\end{thm}
\begin{proof}
     The `if' part follows from the fact that edge ideals of bipartite graphs are Simis, and for any standard linear weighting $w$, the weighted monomial ideal of a Simis ideal is again a Simis ideal (see \cite[Corollary 5.5(a)]{mendez2024symbolic}). For the `only if' part, as the irreducible decomposition of $I$ is minimal, it follows from Lemma~\ref{lem:conj lem 2} that if $I$ is a Simis ideal, then $\alpha_{i,j}= 1$ for all $\{x_i,x_j\}\in E(G(I))$. Hence, the assertion follows from Lemma~\ref{lem:conj lem 1}.
\end{proof}


\section[Simis Property of Support-2 Monomial Ideals]{Simis Property of Support-$2$ Monomial Ideals}\label{sec: small powers}

In this section, we study some particular classes of support-$2$ monomial ideals and try to understand when their symbolic powers coincide with the ordinary powers. In the context of Conjecture~\ref{conj}, it is known that the conjecture does not hold if we drop the condition that the irreducible decomposition is minimal. Thus, one may ask the following general, but astoundingly difficult question: for support-$2$ monomial ideals, what is the statement of Conjecture~\ref{conj} if one drops the condition that the irreducible decomposition is minimal? In the remainder of the article, we attempt to answer this question for specific classes of support-$2$ monomial ideals. Although the overall goal is to characterize Simis ideals, the first step to achieve this is to identify the obstructions for an ideal to be Simis. In this section, we make a detailed study about when an ideal fails to be Simis in the first step, that is, when $I^{(1)}\neq I^1$. This is equivalent to saying that when the ideal $I$ has embedded primes. In the following result, we give a necessary condition on when a support-$2$ monomial ideal has an embedded prime.

\begin{prop}\label{prop:leaf embedded}
    Let $I$ be a support-$2$ monomial ideal and $G(I)$ be its underlying simple graph. Suppose that there is $\{x_i,x_j\} \in E(G(I))$ that satisfies the property that any minimal vertex cover of $G(I)$ either contains $x_i$ or $x_j$, but not both. If there is some minimal monomial generator $u \in \G(I)$ such that $x_i^{\nu_{i,j}+1} \mid u$ or $x_j^{\nu_{j,i}+1}\mid u$, then $I^{(1)} \ne I.$
    \begin{proof}
        Without loss of generality, we can assume that $x_i^{\nu_{i,j}+1}\mid u$. Let us suppose that $u$ corresponds to $\{x_i,x_k\} \in E(G(I))$. We take $f=x_i^{\nu_{i,j}}x_k^{\nu_{k,i}}$. Note that $f\mid u$, $f\neq u$, and so $f\notin I$. Now, let $P$ be any minimal prime ideal of $I$. If $x_i\in P$ then by the given condition, $x_j\notin P$, and therefore, $x_i^{\nu_{i,j}}\in Q(P)$. On the other hand, if $x_i\notin P$, then $x_k\in P$, and so $x_k^{\nu_{k,i}}\in Q(P)$. Thus, in any case $f\in Q(P)$, and hence, $f\in I^{(1)}$. 
    \end{proof}
\end{prop}

Let $G$ be a simple graph, and $x\in V(G)$. The \textit{(open) neighbourhood} of $x$, denoted by $N_G(x)$, is the set $\{ y\in V(G)\mid \{x,y\}\in E(G)\}$. The \emph{degree} of the vertex $x$ is given by $\mathrm{deg}(x)=|N_G(x)|$. A \textit{leaf} in $G$ is a vertex of degree $1$. If $x_i\in V(G)$ is a leaf vertex and $x_j\in N_G(x_i)$, then we say that the edge $\{x_i,x_j\}\in E(G)$ is a leaf edge.

\begin{rmk}\label{remark:one gen}
    We can immediately make the following remarks:
    \begin{itemize}
        \item[(a)] If $\{x_i,x_j\}$ is a leaf edge, then it satisfies the hypothesis of Proposition~\ref{prop:leaf embedded}. 
        \item[(b)] If $\{x_i,x_j\} \in E(G(I))$ satisfies the hypothesis of Proposition~\ref{prop:leaf embedded}, and $I$ does not have an embedded prime, then $\alpha_{i,j} = 1.$
    \end{itemize}
\end{rmk}

The \textit{girth} of a simple graph $G$ is the length of the shortest cycle contained in the graph. For any two vertices $x_i, x_j\in V(G)$, we denote by $d(x_i,x_j)$ the length of the shortest path between $x_i$ and ~$x_j$.
\begin{prop}\label{girth_6_embdd}
    Let $I$ be a support-$2$ monomial ideal and assume that the girth of $G(I)$ is at least six. Let $\{x_i,x_j\} \in E(G(I))$ be such that $\alpha_{i,j} \ge 2$. Suppose that either $d(x_i,x_\ell) \ge 2$ or $d(x_j,x_\ell) \ge 2$ for any leaf vertex $x_\ell \in V(G(I))$, then $I^{(1)} \ne I.$
\end{prop}
\begin{proof}
    Without loss of generality, let us assume that $d(x_i,x_\ell) \ge 2$ for any leaf vertex $x_\ell \in V(G(I))$. It suffices to find a monomial in the first symbolic power of $I$ which is not in $I$. Let us consider the ~monomial
    \[
    f =x_i^{\nu_{i,j}}x_j^{\nu_{j,i}}\prod_{\substack {x_k \in N_{G(I)}(x_i) \setminus \{x_j\}\\ x_r \in N_{G(I)}(x_k)\setminus \{x_i\}}} x_r^{\nu_{r,k}}.
    \]
    We claim that $f\in I^{(1)}\setminus I$. To prove $f\in I^{(1)}$, it is enough to show that for any minimal prime ideal $P$ of $I$, $f\in Q(P)$. There are three possibilities. First, assume that $x_i\in P$ and $x_j\notin P$, then $x_i^{\nu_{i,j}}\in Q(P)$, and thus $f\in Q(P)$. Secondly, if $x_j\in P$ and $x_i\notin P$, by the symmetry we also obtain that $f\in Q(P)$. Finally, if $x_i,x_j\in P$, then since the minimal monomial set of generators of $P$ is a minimal vertex cover of $G(I)$, there is some $x_k\in N_{G(I)}(x_i)$ such that $x_k\notin P$. This implies that $N_{G(I)}(x_k)\subseteq P$, and since $d(x_i,x_\ell) \ge 2$ for any leaf vertex $x_\ell \in V(G(I))$, we have $N_{G(I)}(x_k)\setminus \{x_i\}\neq \emptyset $. Now, take any $x_r\in N_{G(I)}(x_k)\setminus \{x_i\}$, and since, $x_r\in P, x_k\notin P$, we have $x_r^{\nu_{r,k}}\in Q(P)$. Thus, in any case, we have $f\in Q(P)$ and hence, $f\in I^{(1)}$.
    
    It now remains to show that $f\notin I$. Observe that the condition $\alpha_{i,j}\geq 2$ ensures that $x_i^{\nu_{i,j}}x_j^{\nu_{j,i}}\notin I$. Now, it is enough to show that $\{x_i, x_r\}, \{x_j,x_r\}, \{x_r,x_{s}\}\notin E(G(I))$ for any $x_r\in N_{G(I)}(x_{k_1})\setminus \{x_i\}, x_{s}\in N_{G(I)}(x_{k_2})\setminus \{x_i\}$, where $x_{k_1},x_{k_2}\in N_{G(I)}(x_i)\setminus \{x_j\}$. If $\{x_i,x_r\}\in E(G(I))$ for some $x_r\in N_{G(I)}(x_{k_1})$, then, $\{x_i,x_{k_1}, x_r\}$ forms a triangle. Next, if $\{x_j,x_r\}\in E(G(I))$ for some $x_r\in N_{G(I)}(x_{k_1})$, then the induced subgraph of $G(I)$ on the vertices $\{x_i,x_j,x_r,x_{k_1}\}$ has either a $4$-cycle or a triangle. Finally, if $\{x_r,x_{s}\}\in E(G(I))$ where $x_r\in N_{G(I)}(x_{k_1})\setminus \{x_i\}, x_{s}\in N_{G(I)}(x_{k_2})\setminus \{x_i\}$, and $x_{k_1},x_{k_2}\in N_{G(I)}(x_i)\setminus \{x_j\}$, then consider the induced subgraph $H$ of $G(I)$ on the vertices $\{x_i, x_{k_1}, x_{k_2}, x_r,x_{s}\}$. Observe that $H$ always contains a cycle of length $5$ or less, which is a contradiction to the fact that the girth of $G(I)$ is at least six. 
\end{proof}

We are now in a position to characterize support-$2$ Simis ideals whose underlying simple graphs are cycles. 
\begin{thm}\label{thm: Simis cycle}
    Let $I$ be a support-$2$ monomial ideal such that $G(I)=C_n$, where $n\geq 6$. Then, the following conditions are equivalent:
    \begin{enumerate}
        \item $I$ does not have any embedded primes.
        \item $I$ has a standard linear weighting.
    \end{enumerate}
    Moreover, $I^{(s)}=I^s$ for all $s\geq 1$ if and only if $n$ is even and $I$ has a standard linear weighting.
\end{thm}
\begin{proof}
    The implication $(2)\Rightarrow (1)$ follows immediately from \cite[Corollary 5.5(c)]{mendez2024symbolic}, together with the fact that $I(G)$ has no embedded primes for any graph $G$. So we only need to show that $(1)\Rightarrow (2)$. Suppose that $I$ does not have any embedded prime, but there is a vertex $x_i\in V(G(I))$ such that $x_i$ does not have a standard linear weighting. By Proposition~\ref{girth_6_embdd}, it follows that $\alpha_{i,j}\leq 1$ for all $1\leq i,j\leq n$. Since $N_{G(I)}(x_i)=\{x_{i-1}, x_{i+1}\}$, and $x_i$ does not have a standard linear weighting, it follows that $w^1_{i,i-1}\neq w^1_{i,i+1}$. Without loss of generality, assume that $i=2$ and $w^1_{2,1}> w^1_{2,3}$. Now, consider the monomial $f=x_{1}^{w^1_{1,2}}x_2^{w^1_{2,3}}x_{5}^{w^1_{5,4}}$. We claim that $f\in I^{(1)}\setminus I$. To show $f\in I^{(1)}$, it is enough to show that $f\in Q(P)$ for all $P\in \MinAss(I)$. Let $P\in \MinAss(I)$ be such that $x_2\notin P$. Then $x_1\in P$ and $x_1^{w^1_{1,2}}\in Q(P)$. Next, if $x_2\in P$ but $x_3\notin P$, then $x_2^{w^1_{2,3}}\in Q(P)$. Finally, if both $x_2,x_3\in P$, then $x_4\notin P$ as $P$ is a minimal prime of $I$. Thus, $x_5\in P$ and $x_{5}^{w^1_{5,4}}\in Q(P)$. Therefore, $f\in Q(P)$ for any $P\in \MinAss(I)$. It now remains to verify that $f\notin I$. Since $n\geq 6$, it is evident that $f\in I$ if and only if $x_{1}^{w^1_{1,2}}x_2^{w^1_{2,1}}\mid f$. But this is an impossibility, as $w^1_{2,3}< w^1_{2,1}$.
\end{proof}

\begin{rmk}
    We remark that the above characterization requires the assumption that the underlying cycle has length $\geq 6$ and does not always hold for $n\leq 5$ (see Examples~\ref{example:3(1)}, \ref{example:3(2)}, \ref{example:4}, \ref{example:5}, and Section~\ref{sec:discussion} for a detailed discussion).
\end{rmk}

We now turn our attention to support-$2$ monomial ideals whose underlying graph is a whiskered graph. A simple graph $G$ is said to be a \textit{whiskered graph} if there is an induced subgraph $H$ of $G$ such that $G$ can be obtained by putting `whiskers' on each of the vertices of $H$.  More formally, $G$ is a whiskered graph if $V(G)=\{x_1,\ldots , x_m, x_{m+1},\ldots , x_{2m}\}$ and $H$ is an induced subgraph on $\{x_1,\ldots , x_m\}$ such that $E(G)=E(H)\cup \{\{x_1,x_{m+1}\}, \ldots , \{x_m,x_{2m}\}\}$. In this situation, we sometimes write $G=W_H$, and say that $G$ is a whiskered graph on the graph $H$. 

Let $I\subseteq  \K[x_1,\ldots ,x_m,x_{m+1},\ldots , x_{2m}]$ be a support-$2$ monomial ideal such that its underlying graph $G(I)$ is a whiskered graph. We will continue to employ the same notations introduced in Section~\ref{Intro_Pre} for support-$2$ monomial ideals, the only difference being that in this case, the indices $i,j$ will be in $\{1,2,\ldots , 2m\}$. A classical result due to Villarreal \cite[Proposition 2.2]{Villarreal1990} states that the edge ideals of whiskered graphs are always Cohen-Macaulay, and therefore, they have no embedded primes. Moreover, the associated primes of a whiskered graph can be expressed nicely in terms of the structure of the graph. This fact suggests that the support-$2$ monomial ideals whose radical ideal is an edge ideal of a whiskered graph may share some properties. In the next theorem, as a consequence of our study, we give a complete characterization of the Cohen-Macaulayness of support-$2$ monomial ideals whose underlying simple graph is a whiskered graph. This result substantially extends \cite[Theorem 3.1]{HaOriented2019}. For the proof of the following theorem, polarization comes across as an essential technique, so we recall this notion using the notation from \cite{peeva2010graded}. For $S = \mathbb{K}[x_1, \ldots, x_n]$, and a tuple 
$\mathbf{a} = (a_1, \dots, a_n) \in \mathbb{Z}_{\geq 0}^n$, we write $\textbf{x}^{\mathbf{a}}$ for the monomial 
$\textbf{x}^{\mathbf{a}} = \prod_{i= 1}^{n}x_i^{a_i}$ in $S$. 
\begin{defn}\cite[Construction 21.7]{peeva2010graded}
\begin{enumerate}
    \item Let $\textbf{x}^{\mathbf{a}} = x_1^{a_1} \cdots x_n^{a_n}$ be a monomial in $S$. The \emph{polarization} of $\textbf{x}^{\mathbf{a}}$ is defined to be
    \[
   ( \textbf{x}^{\mathbf{a}})^{\pol} = (x_1^{a_1})^{\pol} \cdots (x_n^{a_n})^{\pol},
    \]
    where the operator $(*)^{\pol}$ replaces $x_i^{a_i}$ by a product of distinct variables $\prod_{j=1}^{a_i} x_{i,j}$.
    
    \item Let $I = \left\langle\textbf{x}^{\mathbf{a}_1}, \dots, \textbf{x}^{\mathbf{a}_r}\right\rangle \subseteq S$ be a monomial ideal. The \emph{polarization} of $I$ is defined to be the ~ideal
    \[
    I^{\pol} = \left\langle(\textbf{x}^{\mathbf{a}_1})^{\pol}, \dots, (\textbf{x}^{\mathbf{a}_r})^{\pol}\right\rangle
    \]
    in a new polynomial ring $S^{\pol} = \mathbb{K}[x_{i,j} \mid 1 \leq i \leq n, 1 \leq j \leq p_i]$, where $p_i$ is the maximum power of $x_i$ appearing in $\textbf{x}^{\mathbf{a}_1}, \dots, \textbf{x}^{\mathbf{a}_r}$.
\end{enumerate}
\end{defn}

\begin{thm}\label{thm: CM whisker}
    Let $I\subseteq \K[x_1,\ldots ,x_m,x_{m+1},\ldots , x_{2m}]$ be a support-$2$ monomial ideal such that  $G(I)$ is the whiskered graph $W_H$ on the simple graph $H$, where $V(H)=\{x_1,\ldots , x_m\}$ and, $E(G(I))=E(H)\cup \{\{x_1,x_{m+1}\}, \ldots , \{x_m,x_{2m}\}\}.$
    Then, the following statements are equivalent:
    \begin{enumerate}\rm
        \item $S/I$ is Cohen-Macaulay,
        \item $I$ is unmixed,
        \item $I$ does not have any embedded primes,
        \item For all $1\leq i\leq m$, $\alpha_{i,m+i}=1$ and $w^1_{i,m+i} \geq \mu_{i,j}$ for any $1\leq j\leq m$.
    \end{enumerate}
\end{thm}
\begin{proof}
    The implication $(1)\Rightarrow (2)$ follows from \cite[Corollary 3.1.17]{MR3362802}, and $(2)\Rightarrow (3)$ follows from the definition of unmixedness.
    
    \noindent
    $(3)\Rightarrow (4)$ Note that $\{x_i,x_{m+i}\}\in E(G(I)), 1\leq i\leq m$ are leaves, and therefore satisfy the hypothesis of Proposition~\ref{prop:leaf embedded}. As $I$ does not have any embedded primes, the assertions follow from Proposition~\ref{prop:leaf embedded} and Remark~\ref{remark:one gen}.

    \noindent
    $(4)\Rightarrow (1)$ 
    We apply polarization and construct an Artinian ideal whose polarization is the same as that of $I$. For this purpose, let us consider the ideal $J \subseteq \mathbb{K}[u_1, \ldots, u_m]$ obtained from $I$ by replacing $x_i  = x_{m+i} = u_i , 1 \leq i \leq m$. Thus,
    \[
    J = \left\langle u_i^{w^1_{i,m+i}+w^1_{m+i,i}}, u_i^{w_{i,j}^1}u_j^{w_{j,i}^1}, \ldots,  u_i^{w_{i,j}^{\alpha_{i,j}}}u_j^{w_{j,i}^{\alpha_{i,j}}}: 1 \leq i,j \leq m, \{x_i,x_j\} \in E(H)\right\rangle.
    \]
     
    The assertion $w^1_{i,m+i} \geq \mu_{i,j}$ for any $1\leq j\leq m$ and the fact that $w^1_{m+i,i}\geq 1$ imply that for each $\{x_i,x_j\} \in E(H), 1\leq i,j\leq m$, $u_i^{w_{i,j}^t}u_j^{w_{j,i}^t} \in \mathcal{G}(J)$, where $1 \leq t \leq \alpha_{i,j}$. We now polarize the ideal $J$. In the polarization construction of $J$, the variables $u_1, \ldots, u_m$ are replaced by $u_{1,1}, \ldots, u_{m,1}$ respectively, and
    \begin{align*}
        \left(u_{i}^{w^1_{i,m+i}+w^1_{m+i,i}}\right)^\pol& = \left(\prod_{\ell = 1}^{w^1_{i,m+i}+w^1_{m+i,i}}u_{i,\ell}\right)   \text{ for }\; 1\leq i\leq m,\\
        \left(u_i^{w_{i,j}^t}u_j^{w_{j,i}^t} \right)^\pol &= \left(\prod_{\ell = 1}^{w_{i,j}^t}u_{i,\ell}\right) \left( \prod_{\ell = 1}^{w_{j,i}^t}u_{j,\ell}\right)  \text{ for } 1\leq i,j\leq m, 1\leq t\leq \alpha_{i,j}.\\
    \end{align*}
    
    On the other hand, in the polarization construction of $I$, the variables $x_1, \ldots, x_{2m}$ are replaced by $x_{1,1}, \ldots, x_{2m,1}$, and we have
    \begin{align*}
        \left(x_i^{w^1_{i,m+i}}x_{m+i}^{w^1_{m+i,i}}\right)^\pol &= \left(\prod_{\ell = 1}^{w^1_{i,m+i}} x_{i,\ell }\right) \left(\prod_{\ell = 1}^{w^1_{m+i,i}} x_{m+i,\ell }\right) \text{ for }\; 1\leq i\leq m,\\
        \left(x_i^{w_{i,j}^t}x_j^{w_{j,i}^t} \right)^\pol &= \left(\prod_{\ell = 1}^{w_{i,j}^t}x_{i,\ell}\right) \left( \prod_{\ell = 1}^{w_{j,i}^t}x_{j,\ell}\right)  \text{ for } 1\leq i,j\leq m, 1\leq t\leq \alpha_{i,j}.\\
    \end{align*}

    Let $S = \mathbb{K}[x_1, \ldots, x_{2m}] \text{ and }R = \mathbb{K}[u_1, \ldots, u_m].$ Consider the map
    $\phi : R^\pol \rightarrow S^ \pol$ given by 
    \[
    \phi(u_{i,j}) = \left\{
        \begin{array}{lll}
	       x_{i,j} & \text{ for \;\;}1 \le j \le w^1_{i,m+i},\\
	       x_{(m+i),(j-w^1_{i,m+i})} & \text{ for \;\;} w^1_{i,m+i}+1 \le j \le w^1_{i,m+i}+w^1_{m+i,i}.
        \end{array}\right.
    \] 
    This implies that, $I^\pol = \phi(J^\pol).$
    Observe that the ideal $J$ contains some power of each variable $u_i$ for $1\leq i\leq m$. Therefore, $R/J$ is an Artinian ring and hence $\mathrm{dim}(R/J)=0$. Since $0\leq \mathrm{depth}(R/J)\leq \mathrm{dim}(R/J)$, it follows that $\mathrm{depth}(R/J)=\mathrm{dim}(R/J)$ and hence $R/J$ is Cohen-Macaulay. Consequently, by \cite[Corollary 1.6.3]{herzog2011monomial} $R^\pol/J^\pol$ is Cohen-Macaulay, which in turn implies  $S^\pol/I^\pol$ is Cohen-Macaulay. Therefore, $S/I$ is Cohen-Macaulay again by \cite[Corollary 1.6.3]{herzog2011monomial}. This proves the theorem.
\end{proof}

Finally, we conclude this section with the study of the second symbolic power of support-$2$ monomial ideals, whose underlying simple graph is a whiskered graph. As the primary goal is to find necessary and sufficient conditions when all such ideals are Simis, while studying the second symbolic power of an ideal $I$, it is natural to have the assumption that $I^{(1)}=I$ holds. In the following theorem, under the assumption that the underlying simple graph is triangle-free, we give a complete characterization of when the second symbolic power of such an ideal $I$ coincides with the second ordinary power. 

\begin{thm}\label{thm:whisker second power}
     Let $I$ be a support-$2$ monomial ideal such that $G(I)$ is the whiskered graph $W_H$. Let $V(H)=\{x_1,\ldots , x_m\}$ and $E(G(I))=E(H)\cup \{\{x_1,x_{m+1}\}, \ldots , \{x_m,x_{2m}\}\}$. Further, assume that $I$ does not have any embedded prime, $G(I)$ is triangle-free, and $\alpha_{i,j}\leq 1$ for all $1\leq i,j\leq 2m$. Then, the following statements are equivalent:
     \begin{enumerate}\rm
         \item $ I^{(2)} = I^2$,
         \item For each edge $\{x_i,x_j\} \in E(H)$, one of the following conditions is satisfied:
         \begin{itemize}
             \item $w^1_{i,m+i}=w^1_{i,j}$ and $w^1_{j,m+j}=w^1_{j,i}$,
             \item $w^1_{i,m+i}\geq 2w^1_{i,j}$  and $w^1_{j,m+j} \geq 2w^1_{j,i}$.
         \end{itemize}
     \end{enumerate}
\end{thm}
\begin{proof}
    To simplify the notations, we will write $w_{i,j}$ instead of $w^1_{i,j}$ throughout this proof.

    \noindent
    $(1) \Rightarrow (2)$. Let us assume that there is an edge $\{x_i,x_j\} \in E(H)$ such that $(2)$ does not hold. As $I$ does not have any embedded prime, due to Proposition~\ref{prop:leaf embedded}, we also assume that for each $ 1 \le i \le m$ and $x_j \in N_{G(I)}(x_i)$, $w_{i,m+i}^1 \ge w_{i,j}^1$. By the symmetry of the conditions given in (2), if $(2)$ does not hold, it is straightforward to verify that one of the following is true:

    \begin{enumerate}
        \item[(a)] $w_{i,j} \le w_{i,m+i} < 2w_{i,j}$ and $ w_{j,i} < w_{j,m+j} \le 2w_{j,i}$.
        \item[(b)] $w_{i,j} \le w_{i,m+i} < 2w_{i,j}$ and $ w_{j,m+j} \ge 2w_{j,i}$.
    \end{enumerate}

    First, we assume that $(a)$ holds. We consider the monomial $f = x_i^{w_{i,m+i}}x_j^{2w_{j,i}}x_{m+j}^{w_{m+j,j}}$. We shall show that for any minimal prime ideal $P$ of $I$, $f\in Q(P)^2$. Indeed, if $P$ is a minimal prime of $I$ which contains both $x_i$ and $x_j$, then $x_i^{w_{i,m+i}}, x_j^{w_{j,m+j}} \in Q(P)$. If $P$ is a minimal prime of $I$ containing $x_i$ but does not contain $x_j$, then $x_i^{w_{ij}}, x_{m+j}^{w_{m+j,j}} \in Q(P)$. And if  $P$ is a minimal prime of $I$ containing $x_j$ but not containing $x_i$, then $x_j^{w_{ji}} \in Q(P)$. Note that, in each case $f\in Q(P)^2$, which in turn implies that $f\in I^{(2)}$. Moreover, it is easy to verify that under the conditions given in (a), $f\notin I^2$. Next, if (b) holds, by a similar approach as above, one can verify that $f = x_i^{w_{i,m+i}}x_j^{w_{j,m+j}}x_{m+j}^{w_{m+j,j}}\in I^{(2)}\setminus I^2$.

\medskip
    
    $(2) \Rightarrow (1)$. Assume by way of contradiction that $ I^{(2)} \ne I^2$. Then there is an embedded prime $P'$ of the ideal $I^2$. Note that any minimal prime ideal $P$ of $I$(or $I^2$) has the property that either $x_i\in P$ or $x_{m+i}\in P$. Thus, there is an index $i, 1 \le i \le n$ such that $x_i, x_{m+i} \in P'$. Let $I^2 : v = P'$ for some monomial $v \in S\setminus I^2$. This implies that $x_iv \in I^2$ and $x_{m+i}v \in I^2$. Let $u_1, u_2, v_1,v_2 \in \G(I)$ be such that 
    \begin{equation}\label{eqn:1}
        x_iv = fu_1u_2 \text{ and }
        x_{m+i}v = gv_1v_2,
    \end{equation} 
    where $f, g \in S$ are monomials. We now claim that if $u = u_i = v_j$ for some $1\leq i,j\leq 2$ then $u\nmid v$. 
    
    \noindent
    \textit{Proof of the claim:} If $u \mid v$, then we get $x_i\dfrac{v}{u} \in I$ and $x_{m+i}\dfrac{v}{u} \in I$, and thus $x_i, x_{m+i} \in \left(I : \dfrac{v}{u}\right)$. If $\left(I : \dfrac{v}{u}\right)\neq S$, the variables $x_i, x_{m+i}$ are contained in some associated prime of $I$. Since $I$ does not have any embedded prime, this is an impossibility. Therefore, we obtain $\left(I : \dfrac{v}{u}\right)= S$ and so $\dfrac{v}{u} \in I$. Thus, $v \in I^2$, which is a contradiction to the choice of the monomial $v$. This completes the proof of the ~claim. 

    Moreover, we can assume that $x_i\nmid f \text{ and }x_{m+i} \nmid g$, as otherwise, we have $v\in I^2$. Now, since $x_{m+i} \nmid g$, either $x_{m+i} \mid v_1$ or $x_{m+i} \mid v_2$. Without loss of generality, assume that $x_{m+i} \mid v_1$. Then, $v_1 = x_i^{w_{i,m+i}} x_{m+i}^{w_{m+i,i}}$. Similarly, $x_i \mid u_1$ or $x_i \mid u_2$. Let $x_i \mid u_1$. Then there are two possible choices for $u_1$.
     \begin{enumerate}
         \item[(a)] $u_1 = x_i^{w_{i,m+i}} x_{m+i}^{w_{m+i,i}}$,
         \item[(b)] $u_1 = x_i^{w_{i,j}} x_j^{w_{j,i}}$ for some $\{x_i,x_j \} \in E(H)$.
     \end{enumerate} 
    If $u_1 = x_i^{w_{i,m+i}} x_{m+i}^{w_{m+i,i}}$, then $u_1 \mid x_i v$ implies that $x_{m+i}^{w_{m+i,i}} \mid v$. Moreover, $v_1 \mid x_{m+i} v$ gives $x_i^{w_{i,m+i}} \mid v$. Hence $x_i^{w_{i,m+i}} x_{m+i}^{w_{m+i,i}}\mid v$, where $u_1=v_1=x_i^{w_{i,m+i}} x_{m+i}^{w_{m+i,i}}$, which is a contradiction to the claim above. Therefore, we can conclude that $u_1 = x_i^{w_{i,j}} x_j^{w_{j,i}}$ for some $\{x_i,x_j \} \in E(H)$. Now we have, $v = fx_i^{w_{i,j} - 1} x_j^{w_{j,i}}u_2$. Multiplying both sides by $x_{m+i}$ and by using Equation~\ref{eqn:1}, we obtain
     \begin{align*}
         gx_i^{w_{i,m+i}}x_{m+i}^{w_{m+i,i}}v_2 =x_{m+i}fx_i^{w_{i,j} - 1} x_j^{w_{j,i}}u_2.
     \end{align*}
    From the assumptions given in $(2)$, it follows that $w_{i,m+i}\geq w_{i,j}$, and therefore, dividing both sides of the above equation by $x_i^{w_{i,j}-1}$, we obtain 
    \begin{align*}
        gx_i^{w_{i,m+i} - w_{i,j} + 1}x_{m+i}^{w_{m+i,i}-1}v_2 = f x_j^{w_{j,i}}u_2.
    \end{align*}
    Note that $w_{i,m+i} - w_{i,j} + 1 \ge 1$, and from the fact that $x_i\nmid f$, it follows that $x_i\mid u_2$. Therefore, $u_2 = x_i^{w_{i,k}}x_k^{w_{k,i}}$ for some $\{x_i,x_k\} \in E(G(I))$. Again, by the same argument as above, we have that $k\neq m+i$, and therefore, $\{x_i,x_k\}\in E(H)$. We now have
    \[
        gx_i^{w_{i,m+i} - w_{i,j} + 1}x_{m+i}^{w_{m+i,i}-1}v_2 = f x_j^{w_{j,i}}(x_i^{w_{i,k}}x_k^{w_{k,i}}).
    \]
    
    \noindent
    Dividing both sides of the above equation by $x_i^{w_{i,k}}$, we get
    \begin{equation} \label{consequence}
          \dfrac{gx_i^{w_{i,m+i} - w_{i,j} + 1}x_{m+i}^{w_{m+i,i}-1}v_2}{x_i^{w_{i,k}}} = f x_j^{w_{j,i}}x_k^{w_{k,i}}.
    \end{equation}
    We now consider the following two cases.

    \noindent
    \textsc{Case I:} Assume that $x_i \nmid v_2$. Note that $v_2\nmid f$, as otherwise we have $f\in I$, and hence $v\in I^2$, which is a contradiction to our assumption. This implies that $x_j \mid v_2$, or $x_k \mid v_2$. Without loss of generality, we assume that $x_j\mid v_2$. Let $v_2 = x_j^{w_{j,\ell}}x_\ell^{w_{\ell,j}}$ for some $x_{\ell}\in V(G(I)), \ell \ne i$. First, let us consider the case when $k \ne j$. Replacing $v_2 = x_j^{w_{j,\ell}}x_\ell^{w_{\ell,j}}$ and dividing both sides by $x_j^{w_{j,\ell}}$ in Equation \ref{consequence}, we get
    \begin{equation*}\label{eqn:3.4}
          \dfrac{gx_i^{w_{i,m+i} - w_{i,j} + 1}}{x_i^{w_{i,k}}}x_{m+i}^{w_{m+i,i}-1}x_\ell^{w_{\ell,j}} =  \dfrac{fx_j^{w_{j,i}}}{x_j^{w_{j,\ell}}}x_k^{w_{k,i}}.
    \end{equation*}
    Note that $k\neq \ell$, as otherwise, $\{x_i, x_j, x_k\}$ forms a triangle in $G(I)$. Then, we have that $x_k^{w_{k,i}} \mid g$, and consequently, $x_k^{w_{k,i}} x_i^{w_{i,m+i}}x_{m+i}^{w_{m+i,i} - 1}=x_k^{w_{k,i}}\dfrac{v_1}{x_{m+i}}\mid \dfrac{gv_1}{x_{m+i}}$. From the assumption that $w_{i,m+i} \ge w_{i,k}$, we obtain $ x_i^{w_{i,k}}x_k^{w_{k,i}} \mid \dfrac{gv_1}{x_{m+i}}$, and thus we have $v=\left(\dfrac{gv_1}{x_{m+i}}\right)\cdot v_2 \in I^2$, as $x_i^{w_{i,k}}x_k^{w_{k,i}}\in I$. But this is a contradiction. Next, let us suppose that $k = j$. Then from Equation~\ref{consequence}, replacing $v_2 = x_j^{w_{j,\ell}}x_\ell^{w_{\ell,j}}$, $x_k=x_j$, and dividing both sides by $x_j^{w_{j,\ell}}$, we get
    \begin{equation*}
        \dfrac{gx_i^{w_{i,m+i} - w_{i,j} + 1}}{x_i^{w_{i,j}}} x_{m+i}^{w_{m+i,i}-1}x_\ell^{w_{\ell,j}} = \dfrac{f x_j^{2w_{j,i}}}{x_j^{w_{j,\ell}}}
    \end{equation*}
    If $w_{i,m+i} \ge 2w_{i,j}$, then $x_i \mid f$, which is not possible. Therefore, $w_{i,m+i} = w_{i,j}$ and $w_{j,m+j} = w_{j,i}$. As, $\ell \neq j$, we have $x_\ell^{w_{\ell,j}} \mid f$. This implies that $x_\ell^{w_{\ell,j}}x_i^{w_{i,j}-1}x_j^{w_{j,i}}= x_\ell^{w_{\ell,j}}\dfrac{u_1}{x_i} \mid \dfrac{fu_1}{x_i}$. As $w_{j,m+j} = w_{j,i}$ and since $w_{j,m+j} = w_{j,\ell}$, or $w_{j,m+j} \ge 2w_{j,\ell}$, in either case, we get $x_\ell^{w_{\ell,j}}x_j^{w_{j,\ell}} \mid \dfrac{fu_1}{x_i}$ and thus, $v=\left(\dfrac{fu_1}{x_{i}}\right)\cdot u_2 \in I^2$, which is again a contradiction.

    \noindent
    \textsc{Case II:} Assume that $x_i \mid v_2$. If $v_2 = x_i^{w_{i,j}}x_j^{w_{j,i}}=u_1$ or  $v_2 = x_i^{w_{i,k}}x_k^{w_{k,i}}=u_2$, then the fact $v_2 \mid v$ contradicts the claim above. Therefore, it is sufficient to consider the case when $ v_2 = x_i^{w_{i,\ell}}x_\ell^{w_{\ell,i}}$, where $\ell \notin\{j,k\}$. Substituting $v_2 = x_i^{w_{i,\ell}}x_\ell^{w_{\ell,i}}$ in Equation \ref{consequence}, we have that 
    \begin{equation*} \label{eq:3.6}
          \dfrac{gx_i^{w_{i,m+i} - w_{i,j} + w_{i,\ell}+ 1}x_{m+i}^{w_{m+i,i}-1}x_{\ell}^{w_{\ell,i}}}{x_i^{w_{i,k}}} = f x_j^{w_{j,i}}x_k^{w_{k,i}}.
    \end{equation*}
    Since, $j \ne \ell, m+i, \text{ and } k \ne \ell, m+i$, we conclude that, $x_j^{w_{j,i}}x_k^{w_{k,i}}\mid g$, and consequently, we have that $x_j^{w_{j,i}}x_k^{w_{k,i}}x_i^{w_{i,m+i}}x_{m+i}^{w_{m+i,i}-1}= x_j^{w_{j,i}}x_k^{w_{k,i}}\dfrac{v_1}{x_{m+i}} \mid \dfrac{gv_1}{x_{m+i}}$. As $w_{i,m+i}\geq w_{i,j}$, it follows that $\dfrac{gv_i}{x_{m+i}} \in I$ and therefore, $v \in I^2$, which is again a contradiction. Hence $I^{(2)} = I^2.$
\end{proof}

\section{Some Examples and Concluding Remarks}\label{sec:discussion}

In this section, we discuss some computational evidence that came out while studying the symbolic powers of support-$2$ monomial ideals. Let $I\subseteq S$ be any monomial ideal. There are two seemingly unrelated conditions that have been imposed on the ideal $I$ in Conjecture~\ref{conj}: (1) the irreducible decomposition of $I$ needs to be minimal, and (2) $I$ does not have any embedded primes, or equivalently, $I^{(1)}=I$. The condition (2) is of course, a necessary condition for an ideal $I$ to be Simis, but it is too far away from being sufficient. We also note that there are Simis monomial ideals that do not satisfy the condition (1). 

\begin{exmp}\label{example:1}
    By looking at its primary decomposition, it can be verified that the ideal $I = \left\langle x_1x_2^2,x_2x_3,x_3^2x_4\right\rangle $ is Simis and its minimal primary decomposition is not irreducible. Indeed, the minimal primary decomposition of the ideal is given by $I=\left\langle x_1,x_3 \right\rangle\cap \left\langle x_2, x_4\right\rangle\cap \left\langle x_2^2, x_2x_3, x_3^2\right\rangle$. Set $Q_1=\left\langle x_1,x_3 \right\rangle, Q_2=\left\langle x_2, x_4\right\rangle,$ and $Q_3=\left\langle x_2^2, x_2x_3, x_3^2\right\rangle$. Note that the ideal $Q_3$ is not irreducible. Now we shall show that $I^{(s)} = I^s$ for all $s\geq 1$. We proceed by induction on $s$. The statement for $s=1$ is true. Note that $I^{(s)}=\bigcap_{i=1}^3 Q_i^s$ and take $f \in \G\left(I^{(s)}\right)$. If $x_2x_3\nmid f$, then either $x_2\nmid f$, or $x_3\nmid f$. If $x_2\nmid f$ then since $f\in Q_2^s\cap Q_3^s$, we have $x_4^s\mid f$ and $x_3^{2s}\mid f$. This implies that, $(x_3^2x_4)^s\mid f$ and hence, $f\in I^s$. By the symmetry, if $x_3\nmid f$ then $(x_1x_2^2)^s\mid f$. Now, let us assume that $x_2x_3\mid f$ and set $f'=\frac{f}{x_2x_3}$. Since $f \in Q_3^s$, there exists non-negative integers $c_i, 1\leq i\leq 3$ such that $g=x_2^{2c_1}(x_2x_3)^{c_2}x_3^{2c_3} \mid f$ with $c_1+c_2+c_3= s$. Without loss of generality, assume that $c_1\leq c_3$. Then we can rewrite $g$ as $g=(x_2x_3)^{c_2+2c_1}x_3^{2(c_3-c_1)}$. Thus, if $x_2x_3\mid f$, then $f'=\frac{f}{x_2x_3}\in Q_i^{s-1}$ for all $1\leq i\leq 3$. Therefore, by induction hypothesis, $f'\in I^{s-1}$, and hence, we have $f\in I^s$. Thus, $I^{(s)} = I^s$ for all $s \ge 1$.
\end{exmp}

We further note that, for support-$2$ monomial ideals, the assumption (1) does not imply that $\alpha_{i,j}\leq 1$ for all $\{x_i,x_j\}\in G(I)$. So, the assertions in Lemma~\ref{lem:conj lem 2} are not self-contradictory.

\begin{exmp}\label{example:2}
    The irreducible decomposition of the ideal $I = \left\langle x_1x_2^4,x_2^4x_3,x_2x_3^4,x_3^4x_4\right\rangle$ is minimal. Indeed, the minimal primary decomposition of the ideal is given by $I = \left\langle x_1, x_3\right\rangle \cap \left\langle x_2, x_4\right\rangle \cap \left\langle x_2^4, x_3^4\right\rangle$. Moreover, we have  $x_2^4x_3^4 \in I^{(2)} \setminus I^2$.
\end{exmp}

Next, we consider several examples where the underlying graph of the support-$2$ monomial ideal $I$ is a cycle. Recall that the Simis property has been characterized in Theorem~\ref{thm: Simis cycle} when the cycle has length $\geq 6$. The following examples suggest that the same characterization does not hold for ~$C_5$.
\begin{exmp}\label{example:3(1)} The ideal $I = \left\langle x_1^2x_2,x_2x_3,x_3x_4,x_4x_5,x_5x_1\right\rangle$ can be regarded as an edge ideal of a weighted oriented graph. It can be verified using the notion of strong vertex covers introduced in \cite{PRT2019} that $I$ does not have any embedded prime.
\end{exmp}
        
\begin{exmp}\label{example:3(2)}       
By examining the minimal primary decomposition, one can verify that the ideal $I = \left\langle x_1x_2,x_2x_3^2,x_3x_4,x_4^2x_5,x_5x_1\right\rangle$ is a Simis ideal by showing that every minimal monomial generator of the symbolic power $I^{(s)}$ lies in the ordinary power $I^s$ for every $s \ge 1$. Note that the minimal primary decomposition of $I$ is given by
    \[
    I= \left\langle x_1,x_2,x_4\right\rangle \cap \left\langle x_2,x_3,x_5 \right\rangle  \cap \left\langle x_2,x_4,x_5 \right\rangle \cap \left\langle x_1,x_3,x_5 \right\rangle \cap \left\langle x_1,x_3^2,x_3x_4,x_4^2\right\rangle.
    \]
    We proceed by induction to show that $I^s=I^{(s)}$ for all $s\geq 1$. Since $I$ does not have any embedded primes, the above statement is true for $s=1$. Now assume that $s\geq 2$, and let $f\in \G\left(I^{(s)}\right)$. Then there exists non-negative integers $p_i, q_i, m_i, n_i, \ell_j, 1\leq i\leq 3, 1\leq j\leq 4$, such that the monomials $x_1^{p_1}x_2^{p_2}x_4^{p_3},x_2^{q_1}x_3^{q_2}x_5^{q_3}, x_2^{m_1}x_4^{m_2}x_5^{m_3}, x_1^{n_1}x_3^{n_2}x_5^{n_3}, x_1^{\ell_1}x_3^{2\ell_2}(x_3x_4)^{\ell_3}x_4^{2\ell_4}$ divide $f$, where $\sum_{i}p_i=\sum_{i}q_i=\sum_{i}m_i=\sum_{i}n_i=\sum_{j}\ell_j=s$. First, let us assume that $x_3x_4 \mid f$, and consider $f' = \frac{f}{x_3x_4}$. Note that if $\ell_2\leq \ell_4$, we can rewrite the monomial $x_1^{\ell_1}x_3^{2\ell_2}(x_3x_4)^{\ell_3}x_4^{2\ell_4}= x_1^{\ell_1}(x_3x_4)^{\ell_3+2\ell_2}x_4^{2(\ell_4-\ell_2)}$. This implies that if $x_3x_4 \mid f$, we may assume that $\ell_3\neq 0$. Thus, from the above primary decomposition, it follows that $f' = \frac{f}{x_3x_4}\in I^{(s-1)}$, and by the induction hypothesis, $f'\in I^{s-1}$. This implies that $f\in I^s$, and we are through. Next, we assume that $x_3x_4\nmid f$. In this case, we will directly show that $f\in I^s$. If $x_3, x_4 \nmid f$, then $\ell_2 = \ell_3 =\ell_4=0 $ and $m_2 = 0$. This implies that $\ell_1 = s=m_1+m_3$, and hence $x_1^{s}x_2^{m_1}x_5^{m_3}=(x_1x_2)^{m_1}(x_1x_5)^{m_3}\mid f$. Thus $f \in I^s$. Next, let us assume that $x_3 \mid f$ but $x_4 \nmid f$. Then $p_3 = \ell_3 = \ell_4 = m_2 = 0$, and hence,
    \[
    x_1^{\max\{p_1,\ell_1,n_1\}}x_2^{\max\{p_2,q_1,m_1\}}x_3^{\max\{q_2,2\ell_2,n_2\}}x_5^{\max\{q_3,m_3,n_3\}}\mid f.
    \]
    If $\ell_1 \ge m_3$, then $\ell_2 \le m_1$ and $\ell_1-m_3=m_1-\ell_2\geq 0$. Note that $x_1^{\ell_1}x_2^{m_1}x_3^{2\ell_2}x_5^{m_3}\mid f$, and we can rewrite $x_1^{\ell_1}x_2^{m_1}x_3^{2\ell_2}x_5^{m_3}=(x_1x_5)^{m_3}(x_2x_3^2)^{\ell_2}x_1^{\ell_1-m_3}x_2^{m_1-\ell_2}=(x_1x_5)^{m_3}(x_2x_3^2)^{\ell_2}(x_1x_2)^{\ell_1-m_3}$. Therefore, in this case $f\in I^s$. Next, assume that $\ell_1<m_3$. Then, $\ell_2>m_1$. Further, if $\ell_1\geq p_1$, then $p_2\geq \ell_2>m_1$. Again, note that $x_1^{\ell_1}x_2^{p_2}x_3^{2\ell_2}x_5^{m_3}\mid f$, and we can rewrite this as $x_1^{\ell_1}x_2^{p_2}x_3^{2\ell_2}x_5^{m_3}=(x_1x_5)^{\ell_1}(x_2x_3^2)^{\ell_2}\cdot x_2^{p_2 - \ell_2}x_5^{m_3-\ell_1}$. This shows that $f\in I^s$. Finally, assume that $\ell_1<m_3$ and $\ell_1 < p_1$. Now, let $m_3\leq p_1$. Note that $x_1^{p_1}x_2^{m_1}x_3^{2\ell_2}x_5^{m_3}\mid f$. We can rewrite this as $x_1^{p_1}x_2^{m_1}x_3^{2\ell_2}x_5^{m_3}= (x_1x_5)^{m_3}(x_2x_3^2)^{m_1}\cdot x_1^{p_1-m_3}x_3^{2(\ell_2-m_1)}$. This implies that $f\in I^s$. Lastly, if $m_3>p_1$, then since $\ell_1<p_1$, we have $\ell_2>p_2$. Note that $x_1^{p_1}x_2^{p_2}x_3^{2\ell_2}x_5^{m_3}\mid f$, and we can rewrite this as $x_1^{p_1}x_2^{p_2}x_3^{2\ell_2}x_5^{m_3}=(x_1x_5)^{p_1}(x_2x_3^2)^{p_2}\cdot x_3^{2(\ell_2-p_2)}x_5^{m_3-p_1}$. Therefore, if $x_3\mid f$ but $x_4\nmid f$, then again $f\in I^s$. By the symmetry, if $x_4\mid f$ but $x_3\nmid f$, then one can similarly show that $f\in I^s$, and this completes the proof of our claim.
\end{exmp}

When the underlying simple graph of $I$ is either $C_4$ or $C_5$, then the following statement is true: if $\alpha_{i,j} \ge 2$ for some $i,j$, then $I$ has embedded primes. Let $V(C_n) = \{x_1,\ldots,x_n\}$ and without loss of generality, assume that $i =1, j =2$. If $G(I)=C_4$, consider the monomial $f = x_1^{\nu_{1,2}} x_2^{\nu_{2,1}}$. We will show that $f \in I^{(1)} \setminus I$. The minimal prime ideals of $I$ are given by $P_1=\left\langle x_1,x_3\right\rangle$ and $P_2=\left\langle x_2,x_4\right\rangle$. Note that $x_1^{\nu_{1,2}}\in Q(P_1)$ when $x_2\notin P_1$, and $x_2^{\nu_{2,1}}\in Q(P_2)$ when $x_1\notin P_2$. Thus, $f\in Q(P_1)\cap Q(P_2)=I^{(1)}$, but $f\notin I$ since $\alpha_{1,2}\geq 2$. 

If $G(I)=C_5$, we consider the monomial $g = x_1^{\nu_{1,2}} x_2^{\nu_{2,1}} x_4^{\nu_{4,3}}$, and we claim that $g\in I^{(1)} \setminus I$. Indeed, if $P$ is a minimal vertex cover such that $x_1\notin P$ or $x_2\notin P$, then by similar arguments as before, we see that $g\in Q(P)$. Finally, if $x_1, x_2 \in P$, then $x_3\notin P$ and hence $x_4\in P$. Then also $x_4^{\nu_{4,3}}\in Q(P)$, and hence $g\in Q(P)$. This proves that $g \in I^{(1)}$, and it is clear that $g \notin I$.

Thus, we conclude that for a Simis ideal $I$, where $G(I)=C_4 \text{ or } C_5$, $I$ must be a generalized edge ideal, that is, $\alpha_{i,j}\leq 1$ for all $i,j$. Moreover, if the underlying graph of $I$ is $C_4$, we can further say the following:  if some vertex of $C_4$ does not have a standard linear weighting, then $I$ has embedded primes. Let $V(C_4) = \{x_1,\ldots,x_4\}$ and without loss of generality, let us assume that $w_{1,2}^1 < w_{1,4}^1$. The minimal prime ideals associated with $I_w(C_4)$ are precisely $P_1=\left\langle x_1, x_3\right\rangle$ and $P_2=\left\langle x_2, x_4\right\rangle$. Consider $f = x_1^{w_{1,2}^1}x_4^{w_{4,1}^1}.$ It is clear that $f \in Q(P_1)\cap Q(P_2)=I^{(1)}$. Moreover, $f \notin I$ since $w_{1,2}^1 < w_{1,4}^1$. Hence, $I$ must have an embedded prime. Therefore, we have the following remark.

\begin{rmk}\label{remark: 4 cycle}
    The assertions in Theorem~\ref{thm: Simis cycle} hold for the case when $G(I)=C_4$.
\end{rmk}

The situation is completely different when it comes to the case when $G(I)=C_3$. The following example suggests that there are ideals for which $\alpha_{i,j}\geq 2$, but $I$ does not have any embedded primes.
\begin{exmp}\label{example:4}
    The ideal $I = \left\langle x_1x_2^3,x_2^2x_3,x_2x_3^2,x_3^3x_1 \right\rangle$ has no embedded associated prime.
\end{exmp}

When $I$ is a generalized edge ideal and $G(I)=C_3$, it has been characterized in \cite[Theorem 3.10]{das2024equality} when it is Simis. So, it remains to consider the case when $\alpha_{i,j}\geq 2$ for some $1\leq i, j\leq 3, i\neq j$. The following example suggests that there are Simis ideals of this kind.

\begin{exmp}\label{example:5} The ideals $I = \left\langle x_1x_2^3,x_2^2x_3,x_2x_3^2,x_3^3x_1\right\rangle$ and $J = \left\langle x_1x_2^4,x_2^3x_3,x_2^2x_3^2,x_2x_3^3,x_3^4x_1\right\rangle$ are support-$2$ monomial ideals whose base graph is $C_3$. By direct computations, it can be verified that both the ideals $I$ and $J$ are Simis. Let us prove the assertion for the ideal $J$, and the arguments to show that the ideal $I$ is Simis will be similar. The minimal primary decomposition of $J$ is given by 
    \[
    J=\left\langle x_1,x_2 \right\rangle \cap \left\langle x_1,x_3 \right\rangle \cap \left\langle x_2^4, x_2^3x_3, x_2^2x_3^2, x_2x_3^3, x_3^4 \right\rangle.
    \]
    Set $Q_1=\left\langle x_1,x_2 \right\rangle, Q_2= \left\langle x_1,x_3 \right\rangle, \text{ and } Q_3= \left\langle x_2^4, x_2^3x_3, x_2^2x_3^2, x_2x_3^3, x_3^4 \right\rangle$. We shall show that $J^{(s)}=J^s$ for all $s\geq 1$. We proceed by induction on $s$. Since $J$ does not have any embedded primes, the assertion is true for $s=1$. Assume that, $s\geq 2 $ and take $f\in\G(J^{(s)})$. Then there exists non-negative integers $c_i, 1\leq i\leq 5$ such that $(x_2^4)^{c_1}(x_2^3x_3)^{c_2}(x_2^2x_3^2)^{c_3}(x_2x_3^3)^{c_4}(x_3^4)^{c_5}\mid f$, where $\sum_{i=1}^5c_i=s$. First, we will make a few observations. (i) If $c_1=c_5=0$, then $f\in J^s$. (ii) If $c_1\leq c_5$, we can assume that $c_1=0$, as $(x_2^4)^{c_1}(x_3^4)^{c_5}=(x_2^2x_3^2)^{2c_1}(x_3^4)^{c_5-c_1}$. (iii) By the symmetry, if $c_5\leq c_1$, we can assume that $c_5=0$. (iv) If $c_3\leq c_5$, we can assume that $c_3=0$, as $(x_2^2x_3^2)^{c_3}(x_3^4)^{c_5}=(x_2x_3^3)^{2c_3}(x_3^4)^{c_5-c_3}$. (v) If $c_5\leq c_3$, then we can assume that $c_5=0$ as $(x_2^2x_3^2)^{c_3}(x_3^4)^{c_5}=(x_2x_3^3)^{2c_5}(x_2^2x_3^2)^{c_3-c_5}$. (vi) If $c_2\leq c_5$, we can assume that $c_2=0$, as $(x_2^3x_3)^{c_2}(x_3^4)^{c_5} =(x_2^2x_3^2)^{c_2}(x_2x_3^3)^{c_2}(x_3^4)^{c_5-c_2}$. (vii) If $c_5\leq c_2$, we can assume that $c_5=0$, as $(x_2^3x_3)^{c_2}(x_3^4)^{c_5} =(x_2^3x_3)^{c_2-c_5}(x_2^2x_3^2)^{c_5}(x_2x_3^3)^{c_5}$. We now consider two cases.    
    First, assume that $x_1\nmid f$. In this case, we will directly show that $f\in J^s$. Since $x_1\nmid f$, we have $x_2^s\mid f$. Without loss of generality, assume that $c_1\leq c_5$ and $c_5\neq 0$. Then by (ii), we may assume that $c_1=0$. We then have $x_3^{c_2 + 2c_3 + 3c_4 + 4c_5}\mid f$. We claim that $c_2 + 2c_3 + 3c_4 + 4c_5\geq 3s$. On the contrary, if $c_2 + 2c_3 + 3c_4 + 4c_5<3s$, then $c_2 + 2c_3 + 3c_4 + 4c_5<3(c_2+c_3+c_4+c_5)$, and hence $c_5< 2c_2+c_3$. If $c_5\leq c_3$, then by (v) we may assume that $c_5=0$. Then $c_1=0=c_5$, and hence, we are through by (i). Now, assume that $c_3<c_5$. Then, by (iv), we may assume that $c_3=0$, which implies that $c_5<2c_2$. Now, if $c_5 \leq c_2$, then by (vii), we may further assume that $c_5=0$, and again we are through by (i). On the other hand, if $c_2\leq c_5$, then by (vi), we can take $c_2=0$, which is a contradiction to the fact that $c_5$ is a non-negative integer and $c_5<2c_2$.     
    Next, we consider the case when $x_1\mid f$. Without loss of generality, assume that $c_1\leq c_5$. Then, by (ii) we may further assume that $c_1=0$. Again, if $c_5=0$ then we are through by (i). Suppose that $c_5\neq 0$. We claim that $x_3^{s+3}\mid f$. Note that, $x_3^{c_2+2c_3+3c_4+4c_5}\mid f$ and $c_2+c_3+c_4+c_5=s$. Now $c_2+2c_3+3c_4+4c_5=s+c_3+2c_4+3c_5\geq s+3$ as $c_5\geq 1$. Thus $x_3^{s+3}\mid f$, proving the claim. Now, we set $f'=\frac{f}{x_1x_3^4}$. Note that $f'\in Q_1^{s-1}$, $f'\in Q_2^{s-1}$ as $x_3^{s-1}\mid f'$, and $f'\in Q_3^{s-1}$ as $(x_2^3x_3)^{c_2}(x_2^2x_3^2)^{c_3}(x_2x_3^3)^{c_4}(x_3^4)^{c_5-1}\mid f'$. Thus, $f'\in J^{(s-1)}$, and by the induction hypothesis, $f'\in J^{s-1}$. This implies that $f\in J^s$, as desired.

\end{exmp}

The next result characterizes all support-$2$ monomial ideals for which the graded maximal ideal is an associated prime, in the case when the underlying graph $G(I)$ is a $3$-cycle.
\begin{prop}
    Let $I$ be a support-$2$ monomial ideal such that $G(I)=C_3$ and $\mathfrak{m} = \left\langle x_i,x_j,x_k\right\rangle$ be the unique homogeneous maximal ideal of $\mathbb{K}[x_i,x_j,x_k]$. 
    Then, the following two hold:
\end{prop}
    \begin{enumerate}
        \item[(a)] If $\alpha_{i,j} \ge 2$ and $\alpha_{j,k} \ge 2$, then $\m \in \Ass(I)$.
        \item[(b)] Let $\alpha_{\ell,k} = 1$ for $\ell \in \{i,j\}$ and $\alpha_{i,j} \ge 2$. Then, $\m \notin \Ass(I)$ if and only if $w_{i,k}^1 \ge \mu_{i,j}$ and $w_{j,k}^1 \ge \mu_{j,i}$.
        \end{enumerate}
\begin{proof}
    (a) Let $P_1 = \left\langle x_i,x_j\right\rangle$, $P_2 = \left\langle x_j,x_k\right\rangle$ and $P_3 = \left\langle x_i,x_k\right\rangle $ be the minimal primes and $Q_{P_i}$ be the corresponding primary components. Then, consider the following monomial
    \[
    f = \left\{
        \begin{array}{lll}
	       x_i^{\nu_{i,j}} x_j^{\nu_{j,i}} & \text{ if\; } \nu_{j,k} \le \nu_{j,i}, \\
	       x_j^{\nu_{j,k}}x_k^{\nu_{k,j}} & \text{ if\; } \nu_{j,k} > \nu_{j,i}.
        \end{array}\right.
    \] 
    Now, we have $x_j^{\nu_{j,k}} \in Q({P_1})$, $x_j^{\nu_{j,i}} \in Q({P_2})$ and lastly, $x_i^{\nu_{i,j}}, x_k^{\nu_{k,j}} \in Q({P_3}).$ It is easy to see that $f \in I^{(1)}$. Since $\alpha_{i,j} \ge 2$ and $\alpha_{j,k} \ge 2$, we have $f \notin I$. Hence, $\m \in \Ass(I)$.\\
    (b) $(\Rightarrow)$ For simplicity of notation, we write $w_{i,k}^1 = w_{i,k}$, $w_{j,k}^1 = w_{j,k}, w_{k,i}^1 = w_{k,i}$, $w_{k,j}^1 = w_{k,j}$. Assume by way of contradiction that $w_{i,k} < \mu_{i,j}$. Then, consider the monomial $g = x_i^{\max\{w_{i,k}, \nu_{i,j}\}}x_j^{\nu_{j,i}}$. Since $\alpha_{i,j} \ge 2$ and $w_{i,k} < \mu_{i,j}$, we have $g \in I^{(1)} \setminus I.$ \\
    $(\Leftarrow)$  As discussed in (a), under the assumptions $w_{i,k} \ge \mu_{i,j}$ and $w_{j,k} \ge \mu_{j,i}$, we have
    \begin{align*}
        Q({P_1}) &= \left\langle x_i^{w_{i,k}}, x_i^{w_{i,j}^1}x_j^{w_{j,i}^1}, \ldots,x_i^{w_{i,j}^{\alpha_{i,j}}}x_j^{w_{j,i}^{\alpha_{i,j}}}, x_j^{w_{j,k}}\right\rangle,\\
        Q({P_2}) &= \left\langle x_j^{\nu_{j,i}}, x_k^{w_{k,i}}\right\rangle,\\
        Q({P_3}) &=\left\langle x_i^{\nu_{i,j}}, x_k^{w_{k,j}}\right\rangle.
    \end{align*}
    Let $f \in \G(I^{(1)})$. This implies that, $f \in Q({P_1}) \cap Q({P_2}) \cap Q({P_3}).$ Suppose that $x_i^{w_{i,k}} \mid f$. If $x_j^{\nu_{j,i}} \mid f$, then $x_i^{w_{i,k}} x_j^{\nu_{j,i}}\mid f$. Since $w_{i,k} \ge \mu_{i,j}$, this implies $f \in I$. Next, if $x_k^{w_{k,i}} \mid f$, then $x_i^{w_{i,k}}x_k^{w_{k,i}} \mid f$ and thus $f \in I$. On the other hand, suppose that $x_j^{w_{j,k}} \mid f.$ Since $f \in Q({P_3}),$ if $x_i^{\nu_{i,j}} \mid f$, then we are done due to the assumption that $w_{j,k} \ge \mu_{j,i}$. Next, if $x_k^{w_{k,j}} \mid f$, then $x_k^{w_{k,j}}x_j^{w_{j,k}} \mid f$ and thus $f \in I$. This ensures that $I^{(1)} = I$ and hence, $\mathfrak{m} \notin \Ass(I).$
\end{proof}

\section*{Acknowledgements}
The authors would like to thank the anonymous referees for their careful reading and many valuable suggestions, which have significantly improved the clarity and presentation of the paper. The authors would like to thank Trung Chau and Antonino Ficarra for some helpful comments. The authors acknowledge the use of the computer algebra system Macaulay2 \cite{M2} and the online platform SageMath for testing their computations. The second author acknowledges support from the Infosys Foundation. The third author is supported by the Indian Institute of Technology Jammu, under the Seed Grant (SGR-100014).

\subsection*{Data availability statement} Data sharing is not applicable to this article, as no new data were created or analysed in this study.

\subsection*{Conflict of interest} The authors declare that they have no known competing financial interests or personal
relationships that could have appeared to influence the work reported in this paper.

\bibliography{ref.bib}
\bibliographystyle{alpha}
\end{document}